\documentclass[11pt]{article}
\usepackage[margin=1in]{geometry}
\usepackage{amsmath,amssymb,amsthm,mathtools}
\usepackage{hyperref}
\usepackage{microtype}
\usepackage{tikz}
\usetikzlibrary{fit,calc}
\hypersetup{colorlinks=true,linkcolor=blue,citecolor=blue,urlcolor=blue}

\newtheorem{theorem}{Theorem}[section]
\newtheorem{lemma}[theorem]{Lemma}
\newtheorem{conjecture}[theorem]{Conjecture}

\newtheorem{claim}[theorem]{Claim}

\newcommand{\ex}{\operatorname{ex}}

\title{Exact generalized Tur\'an number of vertex-disjoint paths of length two}
\author{%
\begin{tabular}{cc}
Qi Wu$^{*1}$
&
Long-Tu Yuan$^{\dagger2}$
\end{tabular}\\[0.8em]
\begin{tabular}{@{}c@{}}
\small $^{1}$ School of Mathematics and Statistics, Jiangsu Normal University, Xuzhou, Jiangsu 221116, China\\
\small $^{2}$ School of Mathematical Sciences, Key Laboratory of MEA (Ministry of Education),\\
\small Shanghai Key Laboratory of PMMP, East China Normal University, Shanghai 200241, China
\end{tabular}
}
\date{}

\begin{document}

\maketitle
\begingroup
\renewcommand{\thefootnote}{\fnsymbol{footnote}}
\footnotetext[1]{E-mail: wuqimath@163.com.}
\footnotetext[2]{E-mail: ltyuan@math.ecnu.edu.cn.}
\endgroup

\begin{abstract}
We determine the generalized Tur\'an number
of vertex-disjoint paths of length two and characterize all corresponding
extremal graphs. Our proof combines the Lov\'asz form of the Kruskal--Katona
theorem with a discrete convexity argument. 

\end{abstract}



\medskip

\medskip
\pagestyle{plain}

\section{Introduction}


Classical Tur\'an theory asks how many edges a graph may have while avoiding a fixed subgraph.  The extremal problem for forbidden cliques goes back to
Tur\'an \cite{Turan1941}.  Zykov \cite{Zykov1949} proved the corresponding
clique-counting extension of Tur\'an's theorem, and Erd\H{o}s
\cite{Erdos1962} later obtained related results on the numbers of complete
subgraphs.  Related early work includes Bollob\'as's results \cite{Bollobas1976} on coverings by
complete subgraphs, the extremal subgraph-count results of
Gy\H{o}ri, Pach and Simonovits \cite{GyoriPachSimonovits1991}, and the theorem
of Moon and Moser \cite{MoonMoser1965} on the maximum number of maximal cliques.

Given a graph $H$, a graph $G$ is \emph{$H$-free} if it contains no copy of
$H$ as a subgraph.  Let $K_s$ denote the complete graph on $s$ vertices.  For graphs $T$ and
$F$, let $\ex(n,T,F)$ be the maximum number of copies of $T$ in an
$n$-vertex $F$-free graph.  A newer line of work asks the same question with
edges replaced by copies of another graph.  In this notation, the classical
problem is the case $T=K_2$, while the present paper concerns $T=K_s$.  Alon
and Shikhelman \cite{AlonShikhelman2016} set out this generalized Tur\'an
problem in a general form, and  later papers
\cite{GerbnerPalmer2019,Letzter2019} have obtained 
asymptotic results. For a recent survey, see Gerbner and Palmer
\cite{GerbnerPalmerSurvey2026}. Results of this kind often have a different shape from
the edge case, because a graph can lose many edges while keeping many cliques.

Let $P_3$ denote the path on three vertices, and let $kP_3$ denote the union of $k$ vertex-disjoint copies of $P_3$. A vertex of a graph is \emph{universal} if it is adjacent to every other vertex of the graph. For paths and linear forests there is a long edge-counting history.  Erd\H{o}s
and Gallai \cite{ErdosGallai1959} proved the basic theorem for long paths and cycles.  Gorgol \cite{Gorgol2011} asked for the Tur\'an number of disjoint copies of a graph, and Bushaw and Kettle \cite{BushawKettle2011} solved the problem for several paths when
$n$ is large; in particular they settled the large-$n$ case for $kP_3$.  Yuan and Zhang \cite{YuanZhang2017} then determined
$\ex(n,kP_3)$ and the corresponding extremal graphs for all positive integers
$n$ and $k$.    Their result is a central starting point for the present paper, because even in the edge-counting problem two different extremal mechanisms compete: one construction concentrates almost all edges on only
$3k-1$ vertices, while the other uses $k-1$ universal vertices and a matching on the remaining vertices.
Yuan and Zhang \cite{YuanZhang2021} subsequently studied broad families of
disjoint paths, obtaining exact Tur\'an numbers under suitable parity
conditions and for several additional odd-path cases.

The next natural question is whether the same structural graph survives when edges are replaced by cliques.  This is not a formal consequence of the
edge-counting theorem: a graph may have many fewer edges but still keep many copies of $K_s$, and the two edge-extremal constructions may contribute very
different numbers of large cliques.  Clique-counting versions of related
problems have been studied for cycles \cite{Luo2018}, matchings \cite{Wang2020} and linear forests \cite{ZhangWangZhou2022,ZhuChen2022,ZhuZhangChen2021}.
For disjoint forbidden graphs, one recurrent construction \cite{GerbnerMethukuVizer2019} is to take $k-1$ universal vertices and put a suitable forbidden-free graph on the remaining
vertices.  In the case of $kP_3$, this gives the
join of a $(k-1)$-clique and a maximum matching.  The competing construction
is the disjoint union of a $(3k-1)$-clique and a maximum matching.

For a graph $G$, let $V(G)$ and $E(G)$ denote its vertex set and edge set,
respectively.  Let $N(K_s,G)$ denote the number of
$s$-cliques in a graph $G$. Let $e(G)$ denote the number of edges of $G$. For two vertex-disjoint graphs $G_1$ and $G_2$, let
$G_1\cup G_2$ denote their vertex-disjoint union, and let $G_1+G_2$
denote their join, obtained from $G_1\cup G_2$ by adding all edges between $V(G_1)$ and $V(G_2)$.  For $t\ge 1$, let $M_t$ denote a matching on
$t$ vertices, with one isolated vertex when $t$ is odd.  Thus every
edge-maximal $P_3$-free graph on $t$ vertices is isomorphic to $M_t$.

We now introduce a family of constructions (see Figure~\ref{fig:Hq}).  For
integers $n\ge 3k$ and $0\le q\le k-1$, define
$$H(n,k,q):=K_q+\bigl(K_{3(k-q)-1}\cup M_{n-3k+2q+1}\bigr).$$
If $q=0$, then
$H(n,k,0)=K_{3k-1}\cup M_{n-3k+1}$
and if $q=k-1$, then
$$H(n,k,k-1)=K_{k-1}+\bigl(K_2\cup M_{n-k-1}\bigr)
 \cong K_{k-1}+M_{n-k+1}.$$

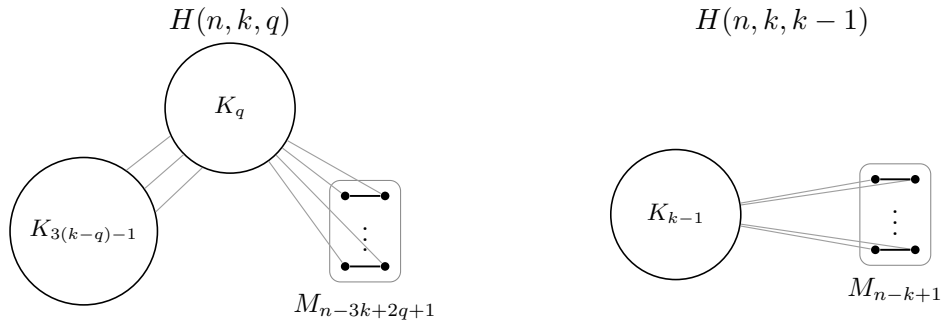
\begin{figure}[ht]
\centering
\begin{tikzpicture}[
  x=0.88cm,y=0.88cm,
  clique/.style={circle,draw=black,fill=white,minimum size=1.72cm,
                 inner sep=1pt,font=\footnotesize,align=center,line width=0.6pt},
  vertex/.style={circle,fill=black,inner sep=1.15pt},
  mbox/.style={draw=black!45,rounded corners=4pt,inner sep=4pt},
  jedge/.style={draw=black!38,line width=0.42pt},
  title/.style={font=\normalsize},
  lab/.style={font=\small}
]
\node[title] at (-4.15,2.15) {$H(n,k,q)$};
\node[clique] (hq) at (-4.15,0.85) {$K_q$};
\node[clique,minimum size=1.95cm] (hc) at (-6.35,-1.00) {$K_{3(k-q)-1}$};

\node[vertex] (hm1) at (-2.42,-0.47) {};
\node[vertex] (hm2) at (-1.82,-0.47) {};
\node[vertex] (hm3) at (-2.42,-1.53) {};
\node[vertex] (hm4) at (-1.82,-1.53) {};
\node[lab] (hdots) at (-2.12,-1.00) {$\vdots$};
\node[mbox,fit=(hm1)(hm2)(hm3)(hm4)(hdots)] (hmbox) {};

\foreach \a/\b in {205/55,225/35,245/15}
  \draw[jedge] (hq.\a)--(hc.\b);
\foreach \v in {hm1,hm2,hm3,hm4}
  \draw[jedge] (hq)--(\v);
\draw[line width=0.72pt] (hm1)--(hm2);
\draw[line width=0.72pt] (hm3)--(hm4);
\node[lab] at (-2.12,-2.13) {$M_{n-3k+2q+1}$};

\node[title] at (4.15,2.15) {$H(n,k,k-1)$};
\node[clique] (gc) at (2.55,-0.75) {$K_{k-1}$};

\node[vertex] (gm1) at (5.55,-0.22) {};
\node[vertex] (gm2) at (6.15,-0.22) {};
\node[vertex] (gm3) at (5.55,-1.28) {};
\node[vertex] (gm4) at (6.15,-1.28) {};
\node[lab] (gdots) at (5.85,-0.75) {$\vdots$};
\node[mbox,fit=(gm1)(gm2)(gm3)(gm4)(gdots)] (gmbox) {};

\foreach \v in {gm1,gm2,gm3,gm4}
  \draw[jedge] (gc)--(\v);
\draw[line width=0.72pt] (gm1)--(gm2);
\draw[line width=0.72pt] (gm3)--(gm4);
\node[lab] at (5.85,-1.88) {$M_{n-k+1}$};
\end{tikzpicture}
\caption{The construction $H(n,k,q)$.}
\label{fig:Hq}
\end{figure}

For $s\ge 2$, write
$$h_s(n,k,q):=N\bigl(K_s,H(n,k,q)\bigr).$$
In particular, since $N(K_2,F)=e(F)$ for every graph $F$,
$$h_2(n,k,0)=e\bigl(H(n,k,0)\bigr),
\qquad
h_2(n,k,k-1)=e\bigl(H(n,k,k-1)\bigr).$$

For a graph $H$, let $\ex(n,H)$ be the maximum number of edges in an
$n$-vertex $H$-free graph. For the edge-counting case, Bushaw and Kettle \cite{BushawKettle2011} determined
$\ex(n,kP_3)$ when $n$ is sufficiently large. Campos and Lopes \cite{CamposLopes2018} and Yuan and Zhang
\cite{YuanZhang2017} independently determined
$\ex(n,kP_3)$ for all positive integers $n$ and $k$.
Moreover, Yuan and Zhang characterized all extremal graphs.
\begin{theorem}\cite{YuanZhang2017}\label{thm:YZ}
Let $n\ge 3k$. Then
$$
\ex(n,kP_3)=\max\{h_2(n,k,0),h_2(n,k,k-1)\}.
$$
Moreover, the extremal graph is $H(n,k,0)$ when
$3k\le n<5k-1$, the extremal graphs are $H(n,k,0)$ and $H(n,k,k-1)$ when
$n=5k-1$, and the extremal graph is $H(n,k,k-1)$ when $n>5k-1$.
\end{theorem}

Chen, Yang, Yuan and Zhang \cite{ChenYangYuanZhang2024} proved exact generalized Tur\'an
results for even linear forests. The first odd path case is already the path
$P_3$, and it is the smallest case in which the two constructions above must
be compared at the level of clique counts.  For $t\ge 0$, let $I_t$ denote
the edgeless graph on $t$ vertices.  Motivated by the preceding edge result,
they proposed the following conjecture.

\begin{conjecture}\cite{ChenYangYuanZhang2024}\label{conj:CYYZ}
Let $n\ge 3k$ and $s\ge 3$. Then
$$\ex(n,K_s,kP_3)=
\max\{h_s(n,k,0),h_s(n,k,k-1)\}.$$
Moreover, if $3\le s\le 3k-1$, then every extremal graph $G$ satisfies either $G\subseteq H(n,k,k-1)$, or
$K_{3k-1}\cup I_{n-3k+1}\subseteq G\subseteq H(n,k,0)$.
\end{conjecture}

Gao, Li, Lu, Sun and Yuan \cite{GaoLiLuSunYuan2026} proved the corresponding
assertion in several cases, including $k=2,3$, $s\ge k+2$, and all sufficiently
large $n$.
In this paper we prove the conjecture for all $n,k$ and all $s\ge 3$.
Since every copy of $K_s$ with $s\ge 3k$ contains $k$ vertex-disjoint copies of $P_3$, every $kP_3$-free graph has no $K_s$ in this range.  Thus the numerical problem is trivial for $s\ge3k$, whereas the nontrivial extremal characterization concerns $3\le s\le3k-1$.  The following is our main result.

\begin{theorem}\label{thm:main}
 Conjecture~\ref{conj:CYYZ} is true.
\end{theorem}

We now give a short outline of the proof of Theorem~\ref{thm:main}.  We first
take an extremal graph which is also edge-maximal among all $kP_3$-free graphs. We show deleting universal vertices, the remaining graph $H$ is $\ell P_3$-critical. The number of $K_s$ in $H$ is bounded by Lemma \ref{lem:local}. The proof of this estimate starts from a triangle bound and then
passes to larger cliques by the Lov\'asz--Kruskal--Katona theorem.  Finally, writing $m=|V(H)|$, Theorem~\ref{thm:YZ} reduces the global problem to two cases, according as
$e(H)\le h_2(m,\ell,0)$ or $e(H)\le h_2(m,\ell,\ell-1)$.  The two resulting algebraic comparisons show that only the two endpoint constructions $H(n,k,0)$ and $H(n,k,k-1)$ can be extremal.

\section{Auxiliary lemmas for \texorpdfstring{$kP_3$}{kP3}-critical graphs}

This section gives the auxiliary lemmas needed for the proof of Theorem~\ref{thm:main}.  Let $\nu(G)$ denote the matching number of $G$. A connected graph $G$ is \emph{$k$-matching-critical} if
$\nu(G-S)<\nu(G)$ for every $S\subseteq V(G)$ with $|S|=k$. A vertex $v$ is a \emph{matching-critical vertex} if $\nu(G-v)<\nu(G)$.  Thus a matching-critical vertex is precisely a vertex
covered by every maximum matching of $G$.

The $k$-matching-critical condition is naturally expressed in terms of the matching deficiency
$\operatorname{def}(G)=|V(G)|-2\nu(G)$, the number of vertices left uncovered by a maximum matching.  In fact, the
definition above is equivalent to $\operatorname{def}(G)<k$.  Berge's
 formula \cite{Berge1958} provided the basic characterization of matching deficiency
through the odd components of vertex-deleted subgraphs.
Gallai \cite{Gallai1964} subsequently developed foundational structural
results on maximum matchings and factor-critical graphs.  Plummer and Saito \cite{PlummerSaito2005} later obtained sharp bounds on
$\operatorname{def}(G)$ for connected  graphs containing no induced copy of $K_{1,r}$, where $K_{1,r}$ is the complete bipartite graph with part sizes $1$ and $r$.  

A matching is  vertex-disjoint copies of $P_2$.
We now consider similar definition for $P_3$.  For a graph $G$, let
$\nu_3(G)=\max\{k:kP_3\subseteq G\}.$
The direct analogue of a matching-critical vertex is a
\emph{$P_3$-critical vertex}: a vertex $v\in V(G)$ such that
$\nu_3(G-v)<\nu_3(G)$.  

For the structural reduction below, we also use a graph-level  notion.
For $t\ge 1$, a graph $G$ is \emph{$tP_3$-critical} if
$\nu_3(G)=t-1$ and $\nu_3(G-v)=t-1$ for every $v\in V(G)$.  Equivalently,
$G$ is $tP_3$-free and $G-v$ contains $t-1$ vertex-disjoint copies of $P_3$
for every $v\in V(G)$.  This graph-level notion is distinct from that of a
$P_3$-critical vertex.  In what follows, a $P_3$-critical vertex will be
called simply a \emph{critical vertex}, or just \emph{critical}; the  term
$tP_3$-critical will be retained for the graph-level notion.

  An $F$-free graph is
\emph{$F$-saturated} if adding any missing edge creates a copy of $F$.

\begin{lemma}\label{lem:promotion}
Let $G$ be a $kP_3$-saturated graph. For $v\in V(G)$, if $v$ is critical, then $v$
is universal.
\end{lemma}

\begin{proof}
Suppose on the contrary that $v$ is not
universal, and choose $u\in V(G)$ such that $uv\notin E(G)$. Since $G$ is
$kP_3$-saturated, $G+uv$ contains $k$ vertex-disjoint copies of $P_3$. At least one of these paths uses the new edge $uv$, and hence contains $v$. The
other $k-1$ paths lie in $G-v$, so $\nu_3(G-v)\ge k-1$. Since $G$ is $kP_3$-free, $\nu_3(G)\le k-1$. Thus
$\nu_3(G-v)=\nu_3(G)=k-1$, a contradiction to the assumption that $v$ is critical.
\end{proof}

The following lemma gives the structural reduction that will be used at the
start of the main proof.

\begin{lemma}\label{lem:residual-critical}
Let $n\ge 3k$, and let $G$ be an $n$-vertex 
$kP_3$-saturated graph. Let $W$ be the set of universal vertices of $G$,
put $q=|W|$, $H=G-W$, and $\ell=k-q$. Then $0\le q\le k-1$,
$G=K_q+H$, and $H$ is $\ell P_3$-critical.
\end{lemma}

\begin{proof}
Clearly, $G=K_q+H$. If $q\ge k$, choose $k$ vertices of $W$ as the
centers of $k$ disjoint copies of $P_3$ and use $2k$ further vertices as their endpoints; this is possible because $n\ge 3k$. This contradicts the
assumption that $G$ is $kP_3$-free. Hence $0\le q\le k-1$.

We first show that $H$ is $\ell P_3$-free. Otherwise, $H$ contains $\ell$
vertex-disjoint copies of $P_3$, using $3\ell=3(k-q)$ vertices. The number of unused vertices of $H$ is
$$|H|-3\ell=(n-q)-3(k-q)=n-3k+2q\ge 2q.$$
Together with the $q$ universal vertices in $W$, these unused vertices form $q$ further disjoint copies of $P_3$, giving a copy of $kP_3$ in $G$, a
contradiction.

Now let $x\in V(H)$. Since $x\notin W$, the vertex $x$ is nonuniversal in
$G$. By Lemma~\ref{lem:promotion}, $x$ is not critical. Since
$G-x$ is a subgraph of $G$, it follows that
$\nu_3(G-x)=\nu_3(G)=k-1$. In a family of $k-1$ vertex-disjoint copies of
$P_3$ in $G-x$, at most $q$ paths meet $W$. Therefore at least
$(k-1)-q=\ell-1$ paths lie entirely in $H-x$. Thus $H-x$ contains
$(\ell-1)P_3$ for every $x\in V(H)$, and consequently $H$ is
$\ell P_3$-critical.
\end{proof}

Denote by $d_G(u)$ the
degree of a vertex $u$ in a graph $G$. We next turn to the counting part of the argument. 

\begin{lemma}\label{lem:charging}
Let $F$ be a graph on $t$ vertices.  For $x\in V(F)$, define $c_F(x)=t+1-d_F(x)$.  Then
$$N(K_3,F)+\sum_{xy\in E(F)}\min\{c_F(x),c_F(y)\}+\sum_{x\in V(F)}\left\lfloor\frac{c_F(x)}2\right\rfloor\le \binom{t+2}{3}.
$$
Moreover, if equality holds and $t\ge 2$, then $F=K_t$.
\end{lemma}

\begin{proof}
Let $\overline F$ be the complement of $F$.  For each vertex $x$, let
$h(x)=d_{\overline F}(x)$ and  $m=e(\overline F)$.
Since $F$ has $t$ vertices, we have $d_F(x)+d_{\overline F}(x)=t-1$.
Therefore $c_F(x)=t+1-d_F(x)=t+1-(t-1-h(x))=h(x)+2$.

Let $M$ be the number of triples of vertices which contain at least one edge of $\overline F$.  A triple of vertices is a triangle in $F$ if and only if it
contains no edge of $\overline F$.  Hence $N(K_3,F)=\binom t3-M$. Also, $e(F)=\binom t2-m$.

Substituting $c_F(x)=h(x)+2$ into the desired inequality, we get
$$\binom t3-M+\sum_{xy\in E(F)}\min\{h(x)+2,h(y)+2\}
+\sum_x\left\lfloor \frac{h(x)+2}{2}\right\rfloor
\le
\binom{t+2}{3}.$$
This is equivalent to
$$\binom t3-M+\sum_{xy\in E(F)}\min\{h(x),h(y)\}+2e(F)
+\sum_x\left\lfloor \frac{h(x)}{2}\right\rfloor
+t\le \binom{t+2}{3}.$$
By $e(F)=\binom t2-m$ and
$\binom{t+2}{3}-\binom t3-t=2\binom t2$,
the preceding inequality is equivalent to
$$\sum_{xy\in E(F)}\min\{h(x),h(y)\}+
\sum_x\left\lfloor \frac{h(x)}{2}\right\rfloor
\le
M+2m.
$$
Thus it remains to prove this last inequality.

For the first term, we have
$\sum_{xy\in E(F)}\min\{h(x),h(y)\}
\le \frac{1}{2}\sum_{xy\in E(F)}(h(x)+h(y))$.
Since the term $h(x)$ appears once for each edge of $F$ incident
with $x$, we have $\sum_{xy\in E(F)}(h(x)+h(y))=\sum_x h(x)d_F(x)$.
So $\sum_{xy\in E(F)}\min\{h(x),h(y)\}
\le \frac{1}{2}\sum_x h(x)d_F(x)$.

  For a fixed vertex
$x$, the number $h(x)d_F(x)$ counts ordered triples $(y,x,z)$ such that
$xy$ is a non-edge of $F$ and $xz$ is an edge of $F$.  Now fix an unordered
triple $T$ of vertices.  If $T$ is a triangle of $F$, then it contributes nothing to this count.  If $T$ contains at least one non-edge of $F$, then,
among the three possible choices of the middle vertex, at most two choices can give one incident non-edge and one incident edge.  Hence only the $M$ triples
which contain a non-edge of $F$ contribute, and each of them contributes at
most two.  Thus $\sum_x h(x)d_F(x)\le 2M$.
Consequently, $\sum_{xy\in E(F)}\min\{h(x),h(y)\}\le M$.

Since $\sum_x h(x)=2e(\overline F)=2m$, we have
$\sum_x\left\lfloor \frac{h(x)}{2}\right\rfloor\le \frac{1}{2}\sum_x h(x)\le m$.
Hence
$$\sum_{xy\in E(F)}\min\{h(x),h(y)\}
+\sum_x\left\lfloor \frac{h(x)}{2}\right\rfloor
\le
M+m
\le
M+2m.$$
This proves the desired inequality.

It remains to discuss the equality case.  If equality holds in the stated
inequality, then equality also holds in the equivalent form
$$\sum_{xy\in E(F)}\min\{h(x),h(y)\}
+\sum_x\left\lfloor \frac{h(x)}{2}\right\rfloor
=M+2m.$$
  Hence $M+2m\le M+m$. So $m=0$.  Thus $e(\overline F)=m=0$.
So $F$ is the complete graph $K_t$.
\end{proof}

The preceding lemma supplies the required triangle estimate.  To extend this
estimate to larger cliques, we use the following
clique-count consequence of the Lov\'asz form \cite{Lovasz1993} of the
Kruskal--Katona theorem \cite{Katona1968,Kruskal1963}.
\begin{theorem}\cite{Katona1968,Kruskal1963,Lovasz1993}\label{thm:lovasz-kk}
Let $G$ be a graph, let $r\ge p\ge 1$ be integers, and let $x\ge r$ be real.
If
$N(K_p,G)\le \binom{x}{p},$
then
$N(K_r,G)\le \binom{x}{r}.$
\end{theorem}

We now combine the structural criticality with the preceding two counting
tools.  Let $G$ be a graph and let $X$ be a subset of $V(G)$. Denote by
$N_{X}(u)$ the set of neighbors of a vertex $u$ in $X$, and write $G[X]$
for the subgraph of $G$ induced by $X$. If $a$ is a nonnegative integer, we put $\binom ab=0$ when $b<0$ or $b>a$.  The next lemma is the main local
estimate.

\begin{lemma}\label{lem:local}
Let $D$ be a connected $(a+1)P_3$-critical graph. Then for every $r\ge 3$,
$N(K_r,D)\le \binom{3a+2}{r}$.
If equality holds for some $r\ge 3$ and the right-hand side is positive, then
$D=K_{3a+2}$.

\end{lemma}

\begin{proof}
If $a=0$, then $D$ is $P_3$-free.  Since $D$ is connected, $D$ is either $K_1$ or $K_2$. So the asserted bound and equality statement hold for every $r\ge 3$.  Now we assume that $a\ge 1$.

By criticality, $D$ contains $a$ vertex-disjoint copies of $P_3$. Choose such copies, and let $S$ be their vertex set.  Then $|S|=3a$.  Let $R=V(D)\setminus S$.  Since $D$ is $(a+1)P_3$-free, $D[R]$ is $P_3$-free. So $D[R]$ consists of a matching and isolated vertices.

  Fix $v\in V(D)$.  Since $D$ is $(a+1)P_3$-critical, $D-v$ contains $a$ disjoint copies of $P_3$. Let $S_v$ be their $3a$ vertices.  If $v$ has two distinct neighbors outside $S_v\cup\{v\}$, then these two neighbors together with $v$ form a $P_3$, so $D$ contains $(a+1)P_3$, a contradiction. Hence $d_D(v)\le 3a+1$ for all $v\in V(D)$.  
  
  Let $F=D[S]$.  For $x\in S$ define $c(x)=3a+1-d_F(x)$.  The degree bound gives $|N_R(x)|\le c(x)$.  We classify triangles of $D$ according to how many vertices they have in $R$.  Triangles entirely in $S$ contribute $N(K_3,F)$.  Triangles with one vertex in $R$ and two in $S$ are counted at most by $\sum_{xy\in E(F)}\min\{c(x),c(y)\}$.
Triangles with two vertices in $R$ must use an edge of the matching in $D[R]$ and a common neighbor in $S$; for each $x\in S$, there are at most $\lfloor c(x)/2\rfloor$ such matching edges.  Hence
$$N(K_3,D)\le N(K_3,F)+\sum_{xy\in E(F)}\min\{c(x),c(y)\}+\sum_{x\in S}\left\lfloor\frac{c(x)}2\right\rfloor.$$
Since $F=D[S]$ has $3a$ vertices, by Lemma~\ref{lem:charging},
$N(K_3,D)\le \binom{3a+2}{3}$.
If $r>3a+2$, then by $d_D(v)\le 3a+1$ for all $v\in V(D)$, $N(K_r,D)=0$.
For $3\le r\le 3a+2$,  
since $N(K_3,D)\le \binom{3a+2}{3}$, by Theorem~\ref{thm:lovasz-kk}, $N(K_r,D)\le \binom{3a+2}{r}$.

Now suppose that equality holds for some $r\ge 3$ with $\binom{3a+2}{r}>0$.  Note that $\binom{x}{3}$ is a continuous and monotonically increasing function. If $N(K_3,D)<\binom{3a+2}{3}$, then there is an $x<3a+2$ such that $N(K_3,D)=\binom{x}{3}$.  Since  $N(K_r,D)=\binom{3a+2}{r}>0$, the graph $D$ contains a $K_r$. Hence $N(K_3,D)\ge \binom{r}{3}$. So $x\ge r$.  By Theorem~\ref{thm:lovasz-kk}, $N(K_r,D)\leq \binom{x}{r}$. Since $x<3a+2$, $N(K_r,D)<\binom{3a+2}{r}$, a contradiction.  Hence $N(K_3,D)=\binom{3a+2}{3}$.
Thus all inequalities above in the triangle count are equalities.  By Lemma~\ref{lem:charging}, $F=K_{3a}$.  Hence $d_F(x)=3a-1$ and $c(x)=2$ for all $x\in S$.

Recall that when  counting triangles with one vertex in $R$ and two in $S$, we use inequality $|N_R(x)\cap N_R(y)|\le \min\{c(x),c(y)\}$. Thus for every edge $xy$ of $F$,  $|N_R(x)\cap N_R(y)|=\min\{c(x),c(y)\}=2$.
Since $|N_R(x)|\le 2$ for every $x\in S$, all vertices of $S$ have the same two neighbors in $R$, say $u$ and $v$.  Since triangles with two vertices in $R$ must use an edge in $D[R]$, $uv\in E(D)$.  Since $D[R]$ is a matching plus isolated vertices and $uv$ is one matching edge, any other vertex of $R$ is not adjacent to any vertex in $S\cup\{u,v\}$, a contradiction to the connectivity of $D$.  Therefore $R=\{u,v\}$ and $D=K_{3a+2}$.
\end{proof}

We will use the following two standard binomial identities (see Merris's textbook \cite{Merris2003}).
\begin{lemma}\cite{Merris2003}\label{lem:binomial-identities}
For integers $n,r$ with $r\ge 1$, we have Pascal's identity
$$\binom{n+1}{r}-\binom{n}{r}=\binom{n}{r-1}.$$
Moreover, for integers $m,r$ with $m\ge r-1$, we have the hockey-stick identity
$$\sum_{j=r-1}^{m}\binom{j}{r-1}=\binom{m+1}{r}.$$
\end{lemma}

These identities yield the elementary inequality below. The non-strict inequality is a direct consequence of the majorization theorem for convex sequences due to Wu and Debnath \cite[Theorem~1]{WuDebnath2007}.
For completeness, and in order to record the precise strictness condition
needed later, we include a short proof.

\begin{lemma}\label{lem:merge}
Let $x,y\ge 2$ and $r\ge 3$ be integers.  Then $\binom xr+\binom yr\le \binom{x+y-2}{r}$.
If $x,y\ge 3$ and $3\le r\le x+y-2$, then the inequality is strict.
\end{lemma}

\begin{proof}
If $y=2$, then
$$\binom{x}{r}+\binom{2}{r}=\binom{x}{r}=\binom{x+2-2}{r},$$
because $r\ge 3$.  Hence the assertion holds.

Now assume that $y\ge 3$.  By Pascal's identity in Lemma~\ref{lem:binomial-identities}, we get
$$\begin{aligned}
\binom{x+y-2}{r}-\binom{x}{r}
&=\sum_{i=0}^{y-3}
\left(
\binom{x+i+1}{r}-\binom{x+i}{r}
\right)=\sum_{i=0}^{y-3}\binom{x+i}{r-1}.
\end{aligned}$$
Since $x\ge 2$, we have $x+i\ge i+2$ for every $0\le i\le y-3$.  Thus
$$\binom{x+i}{r-1}\ge \binom{i+2}{r-1}.$$
So
$$\sum_{i=0}^{y-3}\binom{x+i}{r-1}\ge
\sum_{i=0}^{y-3}\binom{i+2}{r-1}.$$
Let $j=i+2$.  Then
$$\sum_{i=0}^{y-3}\binom{i+2}{r-1}=
\sum_{j=2}^{y-1}\binom{j}{r-1}.$$
Since $r\ge 3$, all terms with $j<r-1$ are zero.  Thus, by the hockey-stick
identity in Lemma~\ref{lem:binomial-identities},
$$\sum_{j=2}^{y-1}\binom{j}{r-1}=
\sum_{j=r-1}^{y-1}\binom{j}{r-1}=
\binom{y}{r}.$$
Combining the above inequalities gives
$$\binom{x+y-2}{r}-\binom{x}{r}\ge \binom{y}{r}.$$
Equivalently,
$$\binom{x}{r}+\binom{y}{r}\le \binom{x+y-2}{r}.$$

If $x,y\ge 3$ and $3\le r\le x+y-2$, then the last summand comparison is strict:
$x+y-3\ge r-1$, while $x+y-3>y-1$.  Hence the whole inequality is strict.
\end{proof}

We shall use the following standard fact on convex sequences.
The estimate
$x_i\leq \frac{d-i}{d}x_0+\frac{i}{d}x_d$
in the first assertion is a special case of
\cite[Theorem~3.1]{Niezgoda2017}.
For completeness, we include a short proof, which also gives
the equality statement.

\begin{lemma}\label{lem:convex-seq}
Let $x_0,\ldots,x_d$ be a real sequence whose first differences
$\Delta_i=x_{i+1}-x_i$ are nondecreasing. Then
$x_i\le \max\{x_0,x_d\}$ for $0\le i\le d$. If the sequence is not
constant, then no internal term $x_i$, $0<i<d$, attains
$\max\{x_0,x_d\}$.
\end{lemma}

\begin{proof}
For $0<i<d$, since
$\Delta_0\le\Delta_1\le\cdots\le\Delta_{d-1}$, we have
$$\frac{x_i-x_0}{i}
=\frac{1}{i}\sum_{j=0}^{i-1}\Delta_j
\le
\frac{1}{d-i}\sum_{j=i}^{d-1}\Delta_j
=\frac{x_d-x_i}{d-i}.$$
Hence
$$x_i\le \frac{d-i}{d}x_0+\frac{i}{d}x_d
\le \max\{x_0,x_d\}.$$
If an internal term $x_i$ attains $\max\{x_0,x_d\}$, then by the
preceding inequalities, we have $x_0=x_i=x_d$. Hence
$\sum_{j=0}^{i-1}\Delta_j=x_i-x_0=0$
and $\sum_{j=i}^{d-1}\Delta_j=x_d-x_i=0.$ 
Since the differences are
nondecreasing,  all the differences are zero. So the
sequence is constant.
\end{proof}

\section{Proof of the main theorem}


Since the  graphs $H(n,k,0)$ and $H(n,k,k-1)$ are both
$kP_3$-free, we have
$$\ex(n,K_s,kP_3)\ge \max\{h_s(n,k,0),h_s(n,k,k-1)\}.$$
Now we prove that $\ex(n,K_s,kP_3)\le \max\{h_s(n,k,0),h_s(n,k,k-1)\}$. Choose an $n$-vertex $kP_3$-free graph $G$ maximizing $N(K_s,G)$; subject to this, choose $G$ with maximum number of edges. Then $G$ is $kP_3$-saturated.

Let $W=\{v\in V(G):d_G(v)=n-1\}$, put $q=|W|$, $H=G-W$, and $\ell=k-q$. By Lemma~\ref{lem:residual-critical}, we have $0\le q\le k-1$, $G=K_q+H$ and $H$ is $\ell P_3$-critical.

Let $D_1,D_2,\ldots,D_c$ be the connected components of $H$. For $1\leq i\leq c$, let $a_i$ be the maximum number of vertex-disjoint copies of $P_3$ in $D_i$.  Since $H$ is $\ell P_3$-critical, it contains $(\ell-1)P_3$ but no $\ell P_3$.  So $\sum_i a_i=\ell-1$.

Moreover, each $D_i$ is $(a_i+1)P_3$-critical.  Indeed, $D_i$ is $(a_i+1)P_3$-free by the definition of $a_i$.  For any $v\in V(D_i)$, the graph $H-v$ contains $(\ell-1)P_3$.  The components $D_j$ with $j\ne i$ contain altogether at most $\sum_{j\ne i}a_j$ disjoint copies of $P_3$, so $D_i-v$ must contain $a_iP_3$.  Thus Lemma~\ref{lem:local} applies to every component.
If $a_i=0$, then $D_i$ is connected and $P_3$-free. So $D_i$ is $K_1$ or $K_2$.

For $r\ge 3$, by Lemma~\ref{lem:local} and Lemma~\ref{lem:merge}, we have
\begin{equation}\label{eq:residual-clique-bound}
N(K_r,H)\le \sum_i\binom{3a_i+2}{r}
   \le \binom{3\sum_i a_i+2}{r}
   =\binom{3\ell-1}{r}.
\end{equation}


  Put $m=|H|=n-q$. Then $m-3\ell=(n-q)-3(k-q)=n-3k+2q\ge 0$.
Since $G=K_q+H$, $N(K_s,G)=\sum_{j=0}^{s}\binom qjN(K_{s-j},H)$.
Here $N(K_0,H)=1$, $N(K_1,H)=m$, and $N(K_2,H)=e(H)$.  By \eqref{eq:residual-clique-bound} and $s\ge 3$, we have
\begin{equation}\label{eq:1}
N(K_s,G)\le \sum_{j=0}^{s-3}\binom qj\binom{3\ell-1}{s-j}+\binom q{s-2}e(H)+\binom q{s-1}m+\binom qs.
\end{equation}
 By Lemma~\ref{lem:residual-critical}, $H$ is $\ell P_3$-free. By Theorem~\ref{thm:YZ},
$e(H)\le \max\{h_2(m,\ell,0),h_2(m,\ell,\ell-1)\}$.  We consider the corresponding two cases.

\subsection[First edge-bound case]{The case  $h_2(m,\ell,\ell-1)\le h_2(m,\ell,0)$}\label{subsec:case1}


By $h_2(m,\ell,\ell-1)\le h_2(m,\ell,0)$,  we have $e(H)\le h_2(m,\ell,0)$. Since
$$h_2(m,\ell,0)=\binom{3\ell-1}{2}
+\left\lfloor\frac{m-3\ell+1}{2}\right\rfloor,$$
by \eqref{eq:1}, we have
$$\begin{aligned}
N(K_s,G)
&\le
\sum_{j=0}^{s-3}\binom qj\binom{3\ell-1}{s-j}  
+\binom q{s-2}
\left[
\binom{3\ell-1}{2}
+
\left\lfloor\frac{m-3\ell+1}{2}\right\rfloor
\right]  \\
&\quad
+\binom q{s-1}m
+\binom qs .
\end{aligned}$$
Note that
$\binom q{s-1}m=\binom q{s-1}\binom{3\ell-1}{1}+(m-3\ell+1)\binom q{s-1}$ and $\binom qs=\binom qs\binom{3\ell-1}{0}$.
So
$$\begin{aligned}
&\sum_{j=0}^{s-3}\binom qj\binom{3\ell-1}{s-j}
+\binom q{s-2}\binom{3\ell-1}{2}
+\binom q{s-1}\binom{3\ell-1}{1}
+\binom qs\binom{3\ell-1}{0}  \\
&\qquad =
\sum_{j=0}^{s}\binom qj\binom{3\ell-1}{s-j}
=
\binom{q+3\ell-1}{s},
\end{aligned}$$
where the last equality is Vandermonde's identity.  Consequently,
$$N(K_s,G)\le
\binom{q+3\ell-1}{s}+(m-3\ell+1)\binom q{s-1}+
\left\lfloor\frac{m-3\ell+1}{2}\right\rfloor
\binom q{s-2}.$$
Since $\ell=k-q$ and $m=n-q$, we have $q+3\ell-1=3k-1-2q$
and $m-3\ell+1=n-3k+1+2q$.
By the definition of $H(n,k,q)$, the last expression is exactly
$h_s(n,k,q)=N(K_s,H(n,k,q))$.  Hence
$N(K_s,G)\le h_s(n,k,q)$, where
$$h_s(n,k,q)=\binom{3k-1-2q}{s}+
(n-3k+1+2q)\binom q{s-1}+
\left\lfloor\frac{n-3k+1+2q}{2}\right\rfloor
\binom q{s-2}.$$

\begin{claim}\label{clm:H-count}
Let $n,k,s$ be integers with $k\ge 1$, $n\ge 3k$, and $s\ge 3$.  For every integer $q$ with $0\le q\le k-1$,
$h_s(n,k,q)\le \max\{h_s(n,k,0),h_s(n,k,k-1)\}$.
Moreover, if $3\le s\le 3k-1$, then every internal $q$, $0<q<k-1$, satisfies the strict inequality.
\end{claim}

\begin{proof}
Let $m_q=3k-1-2q$, $a_q=n-3k+1+2q$ and
$f_q=\left\lfloor a_q/2\right\rfloor$. 
For $0\le q\le k-2$, by Pascal's identity in Lemma~\ref{lem:binomial-identities}, we have
  $$\begin{aligned}
\Delta_q&=h_s(n,k,q+1)-h_s(n,k,q) \\
&=-\binom{m_q-1}{s-1}-\binom{m_q-2}{s-1}
   +2\binom q{s-1}+(a_q+3)\binom q{s-2}+(f_q+1)\binom q{s-3}.
\end{aligned}$$
The complete  calculation, together with
the calculation of the second difference below, is given in
Appendix~\ref{app:h-differences}.
Applying Pascal's identity once more  gives, for
$0\le q\le k-3$,
$$\begin{aligned}
\Delta_{q+1}-\Delta_q
&=\binom{m_q-2}{s-2}+2\binom{m_q-3}{s-2}+\binom{m_q-4}{s-2} \\
&\quad +4\binom q{s-2}+(a_q+6)\binom q{s-3}+(f_q+2)\binom q{s-4}\ge 0.
\end{aligned}$$
Thus the sequence $h_s(n,k,0),\ldots,h_s(n,k,k-1)$ is discrete convex. By Lemma~\ref{lem:convex-seq}, we have
$h_s(n,k,q)\le \max\{h_s(n,k,0),h_s(n,k,k-1)\}$.
If $k\le 2$, then there is no internal $q$.  If $k\ge 3$ and $3\le s\le 3k-1$, then
$\Delta_1-\Delta_0\ge \binom{3k-3}{s-2}>0$. So the sequence is not constant.  Lemma~\ref{lem:convex-seq} gives the strict internal inequality.
\end{proof}

Hence by Claim~\ref{clm:H-count},
$N(K_s,G)\le h_s(n,k,q)\le  \max\{h_s(n,k,0),h_s(n,k,k-1)\}$.

\subsection[Second edge-bound case]{The case  $h_2(m,\ell,0)\le h_2(m,\ell,\ell-1)$}\label{subsec:case2}

By $h_2(m,\ell,0)\le h_2(m,\ell,\ell-1)$, we have $e(H)\le h_2(m,\ell,\ell-1)$.  Since $\ell=k-q$ and
$m=n-q$, we have
$$h_2(m,\ell,\ell-1)=
\binom{k-q-1}{2}+(n-k+1)(k-q-1)
+\left\lfloor\frac{n-k+1}{2}\right\rfloor.$$
Let
$$
\begin{aligned}
g_s(n,k,q)=
&\sum_{i=0}^{s-3}\binom q{i}\binom{3k-3q-1}{s-i} \\
&+\binom q{s-2}
\left[
\binom{k-q-1}{2}+(n-k+1)(k-q-1)
+\left\lfloor\frac{n-k+1}{2}\right\rfloor
\right] \\
&+\binom q{s-1}(n-q)+\binom qs .
\end{aligned}
$$
Then \eqref{eq:1} becomes $N(K_s,G)\le g_s(n,k,q)$.  

For a polynomial $f(x)$, let $[x^s]f(x)$ denote the coefficient of $x^s$ in
$f(x)$.  Recall the binomial convention:
$\binom ab=0$ whenever $a$ is a nonnegative integer and either
$b<0$ or $b>a$.

\begin{claim}\label{clm:fan}
Let $n,k,s$ be integers with $k\ge 1$, $n\ge k-1$, and $s\ge 3$.  Let $$\gamma=\max\{g_s(n,k,0),g_s(n,k,k-1)\}.$$ Then for every
integer $q$ with $0\le q\le k-1$,
$g_s(n,k,q)
\le
\gamma$.
Moreover, if $\gamma>0$, then the inequality is strict for
every $0<q<k-1$.
\end{claim}

\begin{proof} If $q=0$ or $q=k-1$, there is nothing to show. Assume that $0<q<k-1$.
Let $d=k-1$, $h=k-1-q$ and $L=n-k+1$. Then $q=d-h$ and $$g_s(n,k,q)=\sum_{i=0}^{s-3}\binom{d-h}{i}\binom{3h+2}{s-i}
 +\binom{d-h}{s-2}\left[\binom h2+Lh+\left\lfloor\frac L2\right\rfloor\right]
 +\binom{d-h}{s-1}(L+h)+\binom{d-h}{s}.$$

Write $L=2t+\varepsilon$, where $\varepsilon\in\{0,1\}$, and define
\begin{equation}\label{eq:2}
\begin{aligned}
U_h={}&\sum_{i=0}^{s-3}\binom{d-h}{i}\binom{3h+2}{s-i}
+\binom{d-h}{s-2}\left[\binom h2+\varepsilon h\right] +\binom{d-h}{s-1}(h+\varepsilon)+\binom{d-h}{s},\\
b_h={}&2\binom{d-h}{s-1}+(2h+1)\binom{d-h}{s-2}.
\end{aligned}
\end{equation}
Then $g_s(n,k,d-h)=U_h+t b_h$.  

For $0\le h\le d-1$, let
$\Delta_h=g_s(n,k,d-h-1)-g_s(n,k,d-h)$.
Let $A_h=U_{h+1}-U_h$ and $r_h=b_h-b_{h+1}$.  Then
\begin{equation}\label{eq:3}
\Delta_h=A_h-t r_h.
\end{equation}
We first compute $r_h$.  From \eqref{eq:2} and Pascal's identity in Lemma~\ref{lem:binomial-identities}, we get
\begin{equation}\label{eq:4}
\begin{aligned}
r_h
&=2\left[\binom{d-h}{s-1}-\binom{d-h-1}{s-1}\right]
 +(2h+1)\binom{d-h}{s-2}-(2h+3)\binom{d-h-1}{s-2}  \\
&=2\binom{d-h-1}{s-2}
 +(2h+1)\left[\binom{d-h-1}{s-2}+\binom{d-h-1}{s-3}\right]
 -(2h+3)\binom{d-h-1}{s-2}  \\
&=(2h+1)\binom{d-h-1}{s-3}\geq 0.
\end{aligned}
\end{equation}

Define
$$R_h(x)=(1+x)^{3h+2}+(\varepsilon-2h-2)x+
\left[\binom h2+\varepsilon h-\binom{3h+2}{2}\right]x^2.
$$
We now compute $[x^r]R_{h}(x)$. For $r\ge 3$, this is exactly $\binom{3h+2}{r}$.  For $r=1$, the coefficient is
$(3h+2)+(\varepsilon-2h-2)=h+\varepsilon$.
For $r=2$, the coefficient is
$$\binom{3h+2}{2}+\left[\binom h2+\varepsilon h-\binom{3h+2}{2}\right]=\binom h2+\varepsilon h.$$
Hence the coefficients of $R_h(x)$ are
$$[x^r]R_h(x)=
\begin{cases}
1, & r=0,\\
h+\varepsilon, & r=1,\\
\binom h2+\varepsilon h, & r=2,\\
\binom{3h+2}{r}, & r\ge 3.
\end{cases}$$
So the coefficient of $x^s$ in $(1+x)^{d-h}R_h(x)$ is
$[x^s](1+x)^{d-h}R_h(x)=\sum_i \binom{d-h}{i}[x^{s-i}]R_h(x)$.
By the coefficient description of $R_h(x)$ above, this coefficient is exactly
$$\sum_{i=0}^{s-3}\binom{d-h}{i}\binom{3h+2}{s-i}
+\binom{d-h}{s-2}\left[\binom h2+\varepsilon h\right]
+\binom{d-h}{s-1}(h+\varepsilon)
+\binom{d-h}{s},$$
which is $U_h$.  Thus $U_h=[x^s](1+x)^{d-h}R_h(x)$.
So
\begin{equation}\label{eq:6}
\begin{aligned}
A_h
&=U_{h+1}-U_h=[x^s](1+x)^{d-h-1}R_{h+1}(x)
  -[x^s](1+x)^{d-h}R_h(x)  \\
&~~~~~~~~~~~~~~~~~~=[x^s](1+x)^{d-h-1}\left(R_{h+1}(x)-(1+x)R_h(x)\right).
\end{aligned}
\end{equation}

Let $c_h=\binom h2+\varepsilon h-\binom{3h+2}{2}
    =-4h^2+(\varepsilon-5)h-1$.
Then $R_h(x)=(1+x)^{3h+2}+(\varepsilon-2h-2)x+c_hx^2$.
By a direct calculation,
\begin{equation}\label{eq:7}
\begin{aligned}
R_{h+1}(x)-(1+x)R_h(x)
&=(1+x)^{3h+5}-(1+x)^{3h+3}-2x-(6h+7)x^2-c_hx^3.
\end{aligned}
\end{equation}
Here the coefficient $-2$ comes from
$(\varepsilon-2h-4)-(\varepsilon-2h-2)$, and the coefficient $-(6h+7)$ is
$(c_{h+1}-c_h)-(\varepsilon-2h-2)=-(6h+7)$.
By a simple calculation, we have
\begin{equation}\label{eq:8}
(1+x)^{3h+5}-(1+x)^{3h+3}
=x\left((1+x)^{3h+4}+(1+x)^{3h+3}\right).
\end{equation}
The coefficient of $x$ in the right side of \eqref{eq:8} is $2$, which cancels the
term $-2x$ in \eqref{eq:7}.  The coefficient of $x^2$ in the right side of \eqref{eq:8} is
$(3h+4)+(3h+3)=6h+7$, which cancels the term $-(6h+7)x^2$ in \eqref{eq:7}.    The coefficient of $x^3$ is
$\binom{3h+4}{2}+\binom{3h+3}{2}-c_h=13h^2+(23-\varepsilon)h+10$.
Let
$Q_h=13h^2+(23-\varepsilon)h+10.$
Then
\begin{equation}\label{eq:10}
R_{h+1}(x)-(1+x)R_h(x)
=Q_hx^3+\sum_{j\ge 3}\left[\binom{3h+4}{j}+\binom{3h+3}{j}\right]x^{j+1}.
\end{equation}
Here the complete calculation is given in Appendix~\ref{app:R-difference}.
Combining \eqref{eq:6} and \eqref{eq:10}, we get
\begin{equation}\label{eq:11}
A_h=Q_h\binom{d-h-1}{s-3}
 +\sum_{j=3}^{s-1}\left[\binom{3h+4}{j}+\binom{3h+3}{j}\right]
   \binom{d-h-1}{s-1-j}.
\end{equation}
Every term in \eqref{eq:11} is nonnegative, so $A_h\ge 0$.

If $r_h=0$, then by \eqref{eq:3}, $\Delta_h=A_h\ge 0$.  Moreover, by
\eqref{eq:4}, $r_h=0$ means $\binom{d-h-1}{s-3}=0$, and the same is true for every
later index $h'>h$; hence every later difference with $r_{h'}=0$ is also
nonnegative after the same argument.  Thus only the case $r_h>0$ needs a
threshold calculation.

Assume that $r_h>0$ and $r_{h+1}>0$, and define $T_h=A_h/r_h$.
Since $r_h>0$, it follows from \eqref{eq:3} that
\begin{equation}\label{eq:12}
\Delta_h\ge 0 ~\text{if and only if} ~ t\le T_h.
\end{equation}
By \eqref{eq:4} and \eqref{eq:11}, 
\begin{equation}\label{eq:13}
T_h=\frac{Q_h}{2h+1}+
\sum_{j=3}^{s-1}\frac{\binom{3h+4}{j}+\binom{3h+3}{j}}{2h+1}
\cdot
\frac{\binom{d-h-1}{s-1-j}}{\binom{d-h-1}{s-3}}.
\end{equation}
We claim that $T_h$ is nondecreasing on the range where $r_h>0$.

By a direct calculation,
\begin{equation}\label{eq:14}
\frac{Q_{h+1}}{2h+3}-\frac{Q_h}{2h+1}
=\frac{26h^2+52h+16-\varepsilon}{(2h+1)(2h+3)}\ge 0.
\end{equation}
Fix $j\ge 3$.  Now we prove that the two functions
$h\mapsto \binom{3h+3}{j}/2h+1$ and
$h\mapsto \binom{3h+4}{j}/2h+1$
are nondecreasing.    Let $z$ stand for either
$3h+3$ or $3h+4$.  It is enough to prove
\begin{equation}\label{eq:15}
(2h+1)\binom{z+3}{j}\ge (2h+3)\binom zj.
\end{equation}
If $z<j$, then $\binom zj=0$. So \eqref{eq:15} holds.  If $z\ge j$,
then by $j\ge 3$,
\begin{equation}\label{eq:16}
\frac{\binom{z+3}{j}}{\binom zj}
=\frac{(z+1)(z+2)(z+3)}{(z+1-j)(z+2-j)(z+3-j)}
\ge \frac{(z+1)(z+2)(z+3)}{(z-2)(z-1)z}.
\end{equation}
After substituting $z=3h+3$ into the inequality obtained from \eqref{eq:16}, the
difference between the two sides is $3(36h^3+117h^2+117h+34)>0$.
After substituting $z=3h+4$, the corresponding difference is
$3(36h^3+135h^2+153h+46)>0$.  Both expansions are written out in
Appendix~\ref{app:threshold-algebra}.
Thus the two functions
$h\mapsto \binom{3h+3}{j}/2h+1$ and
$h\mapsto \binom{3h+4}{j}/2h+1$
are nondecreasing.

Fix $j$ with $3\le j\le s-1$, and let $h$ be such that
$r_h>0$ and $r_{h+1}>0$. 
\begin{equation}\label{eq:17}
\frac{\binom{d-h-2}{s-1-j}/\binom{d-h-2}{s-3}}{\binom {d-h-1} {s-1-j}/\binom {d-h-1} {s-3}}
=\frac{d-h-s+j}{d-h-s+2}\ge 1.
\end{equation}
  According to \eqref{eq:17},
$h\mapsto \frac{\binom{d-h-1}{s-1-j}}{\binom{d-h-1}{s-3}}$
is nondecreasing.  By \eqref{eq:14}, \eqref{eq:15}-\eqref{eq:17}, every summand in
\eqref{eq:13} is nondecreasing.  Hence $T_h$ is nondecreasing on the range where $r_h>0$.

 Suppose that $\Delta_h\ge 0$.  If $r_h>0$,
then $t\le T_h$ by \eqref{eq:12}.  For every $h'>h$ with $r_{h'}>0$, since  $T_h$ is nondecreasing, $T_{h'}\ge T_{h}$. So $T_{h'}\ge t$. By \eqref{eq:12}, $\Delta_{h'}\ge 0$.  Recall that if a later index has
$r_{h'}=0$, then $\Delta_{h'}\ge 0$.  Thus
\begin{equation}\label{eq:18}
\text{if}~ \Delta_h\ge 0, ~\text{then}~
\Delta_{h'}\ge 0~\text{for every }h'>h.
\end{equation}
 If some difference is nonnegative, let $p$ be the first such index;
otherwise put $p=d$.  Then $\Delta_i<0$ for $0\le i<p$, and by \eqref{eq:18}
$\Delta_i\ge 0$ for $p\le i\le d-1$.
Therefore the sequence
$g_s(n,k,d),g_s(n,k,d-1),\ldots,g_s(n,k,0)$ is weakly decreasing up
to $p$ and weakly increasing afterwards.  Its maximum is attained at an
endpoint. So
$g_s(n,k,d-h)\le \max\{g_s(n,k,d),g_s(n,k,0)\}$ for all $h$.

It remains to prove strictness at internal indices.  
Assume that  
$g_s(n,k,0)\ge g_s(n,k,d)$ and $g_s(n,k,0)>0$.  We show that
\begin{equation}\label{eq:19}
\Delta_{d-1}=g_s(n,k,0)-g_s(n,k,1)>0.
\end{equation}
If $s\ge 4$, then $d-(d-1)=1$, and substitution in the definition of
$g_s(n,k,d-h)$ gives
$g_s(n,k,1)=\binom{3d-1}{s}+\binom{3d-1}{s-1}=\binom{3d}{s}$.
Also $g_s(n,k,0)=\binom{3d+2}{s}$.  Since $g_s(n,k,0)>0$, we have
$\binom{3d+2}{s}>0$, and hence $\binom{3d}{s}<\binom{3d+2}{s}$.
Thus \eqref{eq:19} holds for $s\ge 4$.

Now let $s=3$.  In this case
$$g_s(n,k,d)=\binom d3+L\binom d2+t d
           =\binom d3+t d^2+\varepsilon\binom d2.$$
By  $g_s(n,k,0)\ge g_s(n,k,d)$, we have
\begin{equation}\label{eq:20}
 t\le
 \frac{\binom{3d+2}{3}-\binom d3-\varepsilon\binom d2}{d^2}.
\end{equation}
On the other hand,
$$g_s(n,k,1)=\binom{3d-1}{3}+\binom{d-1}{2}+(2d-1)t+\varepsilon(d-1).$$
The coefficient of $t$ in $g_s(n,k,0)-g_s(n,k,1)$ is $1-2d<0$.  
Hence the difference
$g_s(n,k,0)-g_s(n,k,1)$ is a decreasing  function of $t$.
Let $d'=\frac{\binom{3d+2}{3}-\binom d3-\varepsilon\binom d2}{d^2}$.
Then by \eqref{eq:20}, $t\le d'$.
So
$$g_s(n,k,0)-g_s(n,k,1)\ge \left(g_s(n,k,0)-g_s(n,k,1)\right)\big|_{t=d'}.
$$
The complete substitutions and factorizations 
are given in Appendix~\ref{app:endpoint-algebra}.  They yield
$$g_s(n,k,0)-g_s(n,k,1)\ge
\frac{(d-1)(26d^2-26d-3\varepsilon-4)}{6d}.$$
This lower bound is positive.  Indeed, $d-1>0$ and $6d>0$, while
$26d^2-26d-3\varepsilon-4\ge 26d^2-26d-7$.
For $d\ge 2$, the right-hand side is positive; at $d=2$ it is already
$26\cdot 4-26\cdot 2-7=45>0$, and it increases for larger $d$.  Hence
$g_s(n,k,0)-g_s(n,k,1)>0$.  Thus
$\Delta_{d-1}=g_s(n,k,0)-g_s(n,k,1)>0$ also holds in the case $s=3$.

 Suppose on the contrary that
$g_s(n,k,d-h)=g_s(n,k,0)$ for some $h$.  Then
\begin{equation}\label{eq:21}
0=g_s(n,k,0)-g_s(n,k,d-h)=\sum_{i=h}^{d-1}\Delta_i.
\end{equation}
By \eqref{eq:19}, the last term $\Delta_{d-1}$ is positive.  Let $p$ be the first index such that $\Delta_{p}\ge 0$. If $h\ge p$, then every term in \eqref{eq:21} is nonnegative by \eqref{eq:18} and the last one is
positive, a contradiction.  Hence $h<p$.  Then all differences  $\Delta_{h'}$ with $h'\le h$ are
negative, and so $g_s(n,k,d-h)-g_s(n,k,d)=\sum_{i=0}^{h-1}\Delta_i<0$.
This gives $g_s(n,k,d)>g_s(n,k,d-h)=g_s(n,k,0)$, a contradiction to the assumption
$g_s(n,k,0)\ge g_s(n,k,d)$.  Therefore no internal $q$ satisfies
$g_s(n,k,q)=g_s(n,k,0)$ when $g_s(n,k,0)\ge g_s(n,k,k-1)$ and
$g_s(n,k,0)>0$.

Now  assume that $g_s(n,k,d)>g_s(n,k,0)$, and suppose that
$g_s(n,k,d-h)=g_s(n,k,d)$ for some $0<h<d$.  Then
$0=g_s(n,k,d-h)-g_s(n,k,d)=\sum_{i=0}^{h-1}\Delta_i$.
So there is some $i<h$ such that $\Delta_i\ge 0$.  Applying
\eqref{eq:18} to this index $i$ shows that every $\Delta_j$ with
$j\ge h$ is nonnegative.  Hence
$g_s(n,k,0)-g_s(n,k,d-h)=\sum_{j=h}^{d-1}\Delta_j\ge 0$,
which contradicts $g_s(n,k,d)>g_s(n,k,0)$.  Thus no internal $q$ satisfies
$g_s(n,k,q)=g_s(n,k,k-1)$ when $g_s(n,k,k-1)>g_s(n,k,0)$.  
This completes the proof of Claim \ref{clm:fan}.
\end{proof}

If $q=0$, then
$$g_s(n,k,0)=\binom{3k-1}{s}=h_s(n,k,0).$$
If $q=k-1$, then by
$n-k+1=2t+\varepsilon$, we have
$$
\begin{aligned}
g_s(n,k,k-1)
&=\binom{k-1}{s}
 +(n-k+1)\binom{k-1}{s-1}
 +\left\lfloor\frac{n-k+1}{2}\right\rfloor
  \binom{k-1}{s-2}=h_s(n,k,k-1),
\end{aligned}
$$
which is exactly the number of $K_s$'s in
$K_{k-1}+M_{n-k+1}\cong H(n,k,k-1)$.
Thus, in the case $h_2(m,\ell,0)\le h_2(m,\ell,\ell-1)$, by Claim~\ref{clm:fan}, 
$N(K_s,G)\le \max\{h_s(n,k,0),h_s(n,k,k-1)\}.$
Combining the two cases $h_2(m,\ell,\ell-1)\le h_2(m,\ell,0)$ and  $h_2(m,\ell,0)\le h_2(m,\ell,\ell-1)$, we get 
$N(K_s,G)\le \max\{h_s(n,k,0),h_s(n,k,k-1)\}$.

\subsection{Characterization of the extremal graphs}\label{subsec:extremal}

Now we characterize the extremal graphs. Assume that $n\ge 3k$ and $3\le s\le 3k-1$.  Let
$$
\mathcal E_0=\left\{G:K_{3k-1}\cup I_{n-3k+1}\subseteq G
\subseteq H(n,k,0)\right\}
$$
and $\mathcal E_1=\left\{G:G\subseteq H(n,k,k-1)\right\}$.
We prove that every extremal graph belongs to $\mathcal E_0\cup\mathcal E_1$.

Let $G$ be extremal.  Add edges, without creating a copy of $kP_3$, until an
edge-maximal extremal graph $G^*$ is obtained.  Then
$G\subseteq G^*$ and $N(K_s,G)=N(K_s,G^*)$.
Write
$G^*=K_q+H$, where $0\le q\le k-1$ and $\ell=k-q$.  Let
$\gamma=\max\{h_s(n,k,0),h_s(n,k,k-1)\}$.
We first exclude the internal values $0<q<k-1$.  If
$h_2(m,\ell,0)\ge h_2(m,\ell,\ell-1)$, then Claim~\ref{clm:H-count} gives
$N(K_s,G^*)\le h_s(n,k,q)<\gamma$.
If $h_2(m,\ell,\ell-1)\ge h_2(m,\ell,0)$, then, by
Claim~\ref{clm:fan} and the identities
$g_s(n,k,0)=h_s(n,k,0)$ and $g_s(n,k,k-1)=h_s(n,k,k-1)$, we have
$$N(K_s,G^*)\le g_s(n,k,q)<\max\{g_s(n,k,0),g_s(n,k,k-1)\}=\gamma.
$$
In either case $N(K_s,G^*)<\gamma$, contradicting extremality.  Hence
$q\in\{0,k-1\}$.

If $q=k-1$, then $\ell=1$ and $H$ is $P_3$-free.  Edge-maximality gives
$H=M_{n-k+1},$
for otherwise we could add an edge inside $H$ and still keep $H$ $P_3$-free.
Thus
$G^*=K_{k-1}+M_{n-k+1}\cong H(n,k,k-1)$.
 So $G\in\mathcal E_1$.

If $q=0$, then $\ell=k$ and $G^*=H$. Since $G^*$ is extremal,
$$\gamma=N(K_s,H)
\leq \sum_i\binom{3a_i+2}{s}
\leq \binom{3k-1}{s}
=h_s(n,k,0)
\leq \gamma.$$
Thus equality holds throughout this chain.  We claim that exactly one $a_i$ is
positive.  Otherwise, let 
$a_{i_1},\ldots,a_{i_t}$ be positive with $t\ge2$, and put
$c_j=3\sum_{h=1}^j a_{i_h}+2$.  Let $j\ge2$ be the least index with
$c_j\ge s$.  Then $c_{j-1}\ge5$, $3a_{i_j}+2\ge5$, and
$s\le c_{j-1}+(3a_{i_j}+2)-2$.  By
Lemma~\ref{lem:merge}, 
$$\sum_i\binom{3a_i+2}{s}<\binom{3\sum_i a_i+2}{s}
=\binom{3k-1}{s},$$
a contradiction.  Hence there is a unique index $i_0$ with
$a_{i_0}=k-1$.  By Lemma~\ref{lem:local},
$D_{i_0}=K_{3k-1}$, while every other component is $K_1$ or $K_2$.  Therefore
$H=K_{3k-1}\cup F$, where $F$ is a matching together with possibly some
isolated vertices.

If $F\ne M_{n-3k+1}$, then $F$ has two isolated vertices. Adding the
edge between them leaves $F$ $P_3$-free, and hence leaves
$K_{3k-1}\cup F$ $kP_3$-free, contrary to the edge-maximality of $G^*$.
Thus $F=M_{n-3k+1}$ and $G^*=H(n,k,0)$.

Since $G\subseteq G^*$ and $N(K_s,G)=N(K_s,G^*)$, no edge of
$K_{3k-1}$ can be absent from $G$, because every such edge lies in
$\binom{3k-3}{s-2}>0$ copies of $K_s$.  An edge of $M_{n-3k+1}$, on the
other hand, lies in no copy of $K_s$ because $s\ge3$.  Hence
$K_{3k-1}\cup I_{n-3k+1}\subseteq G\subseteq H(n,k,0)$,
so $G\in\mathcal E_0$.
This completes
the proof of Theorem~\ref{thm:main}.\qed

\clearpage
\appendix
\section{Detailed algebraic derivations}\label{app:algebra}

For ease of verification, this appendix records the algebraic steps that are
compressed in the main proof.  We use throughout the convention
$\binom ab=0$ whenever $a$ is a nonnegative integer and either $b<0$ or
$b>a$.

\subsection{The first and second differences of \texorpdfstring{$h_s(n,k,q)$}{h(s,n,k,q)}}\label{app:h-differences}

By the definitions above, we have
$m_{q+1}=m_q-2$, $a_{q+1}=a_q+2$ and $f_{q+1}=f_q+1$.
Thus
$$\begin{aligned}
\Delta_q
&=h_s(n,k,q+1)-h_s(n,k,q)=\binom{m_q-2}{s}-\binom{m_q}{s} +(a_q+2)\binom{q+1}{s-1}-a_q\binom q{s-1}\\
&\quad +(f_q+1)\binom{q+1}{s-2}-f_q\binom q{s-2}.
\end{aligned}$$
By Pascal's identity  in Lemma \ref{lem:binomial-identities},
$$\binom{m_q-2}{s}-\binom{m_q}{s}
=-\binom{m_q-1}{s-1}-\binom{m_q-2}{s-1}.$$
For the $a_q$-terms, Pascal's identity gives
$$\begin{aligned}
&(a_q+2)\binom{q+1}{s-1}-a_q\binom q{s-1}=(a_q+2)
 \left(\binom q{s-1}+\binom q{s-2}\right)
 -a_q\binom q{s-1}\\
&\quad\quad\quad~~~~~~~~~~~~~~~~~~~~~~~~~~~~~~~=2\binom q{s-1}+(a_q+2)\binom q{s-2}.
\end{aligned}$$
Similarly,
$$\begin{aligned}
&(f_q+1)\binom{q+1}{s-2}-f_q\binom q{s-2}=(f_q+1)
 \left(\binom q{s-2}+\binom q{s-3}\right)
 -f_q\binom q{s-2}\\
&\quad~~~~~~~~~~~~~~~~~~~~~~~~~~~~~~~~~~~~=\binom q{s-2}+(f_q+1)\binom q{s-3}.
\end{aligned}$$
Hence
$$\Delta_q=-\binom{m_q-1}{s-1}-\binom{m_q-2}{s-1}
+2\binom q{s-1}+(a_q+3)\binom q{s-2}
+(f_q+1)\binom q{s-3}.$$

We now calculate the second difference.  Replacing $q$ by $q+1$ in the
preceding formula gives
$$\begin{aligned}
\Delta_{q+1}
&=-\binom{m_q-3}{s-1}-\binom{m_q-4}{s-1}
 +2\binom{q+1}{s-1} +(a_q+5)\binom{q+1}{s-2}
 +(f_q+2)\binom{q+1}{s-3}.
\end{aligned}$$
For the terms involving $m_q$, Pascal's identity gives
$$\begin{aligned}
&\binom{m_q-1}{s-1}+\binom{m_q-2}{s-1}
 -\binom{m_q-3}{s-1}-\binom{m_q-4}{s-1}=\binom{m_q-2}{s-2}+2\binom{m_q-3}{s-2}+\binom{m_q-4}{s-2}.
\end{aligned}$$
The remaining terms satisfy
$$2\left(\binom{q+1}{s-1}-\binom q{s-1}\right)
=2\binom q{s-2},$$
$$\begin{aligned}
&(a_q+5)\binom{q+1}{s-2}-(a_q+3)\binom q{s-2}=2\binom q{s-2}+(a_q+5)\binom q{s-3},
\end{aligned}$$
and
$$\begin{aligned}
&(f_q+2)\binom{q+1}{s-3}-(f_q+1)\binom q{s-3}=\binom q{s-3}+(f_q+2)\binom q{s-4}.
\end{aligned}$$
Hence
$$\begin{aligned}
\Delta_{q+1}-\Delta_q
&=\binom{m_q-2}{s-2}+2\binom{m_q-3}{s-2}
  +\binom{m_q-4}{s-2}+4\binom q{s-2}+(a_q+6)\binom q{s-3}+(f_q+2)\binom q{s-4}.
\end{aligned}$$

\subsection{The polynomial difference in Claim~\ref{clm:fan}}\label{app:R-difference}

Recall that
$$c_h=\binom h2+\varepsilon h-\binom{3h+2}{2}.$$
Expanding the binomial coefficients gives
$$\begin{aligned}
c_h
&=\frac{h(h-1)}2+\varepsilon h
  -\frac{(3h+2)(3h+1)}2=-4h^2+(\varepsilon-5)h-1.
\end{aligned}$$
Since
$R_h(x)=(1+x)^{3h+2}+(\varepsilon-2h-2)x+c_hx^2,$
we have
$$\begin{aligned}
(1+x)R_h(x)
&=(1+x)^{3h+3}+(\varepsilon-2h-2)x +(c_h+\varepsilon-2h-2)x^2+c_hx^3,
\end{aligned}$$
whereas
$R_{h+1}(x)=(1+x)^{3h+5}+(\varepsilon-2h-4)x+c_{h+1}x^2.$
Furthermore,
$c_{h+1}-c_h=-8h+\varepsilon-9,$
and hence
$(c_{h+1}-c_h)-(\varepsilon-2h-2)=-(6h+7).$
So
$$\begin{aligned}
R_{h+1}(x)-(1+x)R_h(x)
&=(1+x)^{3h+5}-(1+x)^{3h+3}-2x-(6h+7)x^2-c_hx^3.
\end{aligned}$$
Also,
$$\begin{aligned}
(1+x)^{3h+5}-(1+x)^{3h+3}
&=(1+x)^{3h+3}\bigl((1+x)^2-1\bigr)=x\bigl((1+x)^{3h+4}+(1+x)^{3h+3}\bigr).
\end{aligned}$$
The coefficients of $x$ and $x^2$ in the last expression are $2$ and
$(3h+4)+(3h+3)=6h+7$, so they cancel the corresponding negative terms.
The coefficient of $x^3$ after cancellation is
$$\begin{aligned}
&\binom{3h+4}{2}+\binom{3h+3}{2}-c_h=9h^2+18h+9-\bigl(-4h^2+(\varepsilon-5)h-1\bigr)\\
&\quad~~~~~~~~~~~~~~~~~~~~~~~~~~~~~~~~~~=13h^2+(23-\varepsilon)h+10=Q_h.
\end{aligned}$$
  Consequently,
$$R_{h+1}(x)-(1+x)R_h(x)
=Q_hx^3+
\sum_{j\ge3}
\left(\binom{3h+4}{j}+\binom{3h+3}{j}\right)x^{j+1}.$$

\subsection{Threshold and binomial-ratio calculations}\label{app:threshold-algebra}

To prove \eqref{eq:15}, it is enough to show that $$\frac{(z+1)(z+2)(z+3)}{(z-2)(z-1)z}
\ge
\frac{2h+3}{2h+1}.$$
For the
case $z=3h+3$,
$$
\begin{aligned}
&(2h+1)(3h+4)(3h+5)(3h+6) -(2h+3)(3h+1)(3h+2)(3h+3)\\&=3(36h^3+117h^2+117h+34)>0,
\end{aligned}
$$
and for the case $z=3h+4$,
$$
\begin{aligned}
&(2h+1)(3h+5)(3h+6)(3h+7) -(2h+3)(3h+2)(3h+3)(3h+4)\\
&=3(36h^3+135h^2+153h+46)>0.
\end{aligned}
$$
This proves the two monotonicity assertions used there.

\subsection{Endpoint substitutions and the \texorpdfstring{$s=3$}{s=3} factorization}\label{app:endpoint-algebra}
If $q=0$, then $g_s(n,k,0)=\binom{3d+2}{s}$.
  Write
$L=2t+\varepsilon$, where $\varepsilon\in\{0,1\}$.  Direct substitution gives
$$\begin{aligned}
g_3(n,k,d)
&=\binom d3+L\binom d2+td=\binom d3+td^2+\varepsilon\binom d2,
\end{aligned}$$
while
$g_3(n,k,0)=\binom{3d+2}{3}$
and
$$g_3(n,k,1)=\binom{3d-1}{3}+\binom{d-1}{2}
 +(2d-1)t+\varepsilon(d-1).$$
The assumption $g_3(n,k,0)\ge g_3(n,k,d)$ gives
$$t\le d'=\frac{\binom{3d+2}{3}-\binom d3-
\varepsilon\binom d2}{d^2}.$$
Using
$$\binom{3d+2}{3}=\frac92d^3+\frac92d^2+d,
\quad
\binom d3=\frac16d^3-\frac12d^2+\frac13d,
\quad
\binom d2=\frac12d^2-\frac12d,$$
we obtain
$$d'=\frac{26d^2-3d\varepsilon+30d+3\varepsilon+4}{6d}.
$$
Moreover,
$$\binom{3d+2}{3}-\binom{3d-1}{3}-\binom{d-1}{2}
=13d^2-3d.$$
Let
$D(t)=g_3(n,k,0)-g_3(n,k,1)$. Then
$D(t)=13d^2-3d-(2d-1)t-\varepsilon(d-1).$
Since $D(t)$ is decreasing and $t\le d'$, we have $D(t)\ge D(d')$.  Finally,
$$\begin{aligned}
D(d')
&=\frac{1}{6d}\Bigl(
 6d\bigl(13d^2-3d-\varepsilon(d-1)\bigr)-(2d-1)
 (26d^2-3d\varepsilon+30d+3\varepsilon+4)
 \Bigr)\\
&=\frac{26d^3-52d^2-3d\varepsilon+22d+3\varepsilon+4}{6d}\\
&=\frac{(d-1)(26d^2-26d-3\varepsilon-4)}{6d},
\end{aligned}
$$
which is the factorization used in the main proof.

\end{document}